\documentclass[12pt]{amsart}

\usepackage{amsfonts,amsthm,amsmath,amssymb,amscd,mathrsfs,tikz}
\usepackage{graphics}
\usepackage{indentfirst}
\usepackage{cite}
\usepackage{latexsym}
\usepackage[dvips]{epsfig}
\usepackage{lmodern}
\usepackage{color}
\usepackage{amssymb,amsmath,bm}
\usepackage{fontspec}

\usepackage[title]{appendix}
\usepackage{enumerate}

\newtheorem{theorem}{Theorem}[section]
\newtheorem{remark}{Remark}[section]
\newtheorem{definition}{Definition}[section]
\newtheorem{lemma}[theorem]{Lemma}
\newtheorem{pro}[theorem]{Proposition}

\newcommand{\bt}{\begin{theorem}}
	\newcommand{\bl}{\begin{lemma}}
		\newcommand{\el}{\end{lemma}}
	\newcommand{\et}{\end{theorem}}

\newcommand{\bn}{\begin{eqnarray}}
	\newcommand{\en}{\end{eqnarray}}
\newcommand{\bnn}{\begin{eqnarray*}}
	\newcommand{\enn}{\end{eqnarray*}}
\newcommand{\R}{\mathbb{R}}
\newcommand{\ba}{\begin{aligned}}
	\newcommand{\ea}{\end{aligned}}
\newcommand{\be}{\begin{equation}}
	\newcommand{\ee}{\end{equation}}

\newcommand{\p}{\partial}

\newcommand{\dd}{\mathrm{d}}
\newcommand{\Bv}{{\boldsymbol{v}}}
\newcommand{\BU}{\boldsymbol{U}}
\newcommand{\bBV}{\boldsymbol{V}}
\newcommand{\bBW}{\boldsymbol{W}}

\newcommand{\Bu}{{\boldsymbol{u}}}
\newcommand{\Be}{{\boldsymbol{e}}}
\newcommand{\BF}{{\boldsymbol{F}}}

\newcommand{\Bx}{{\boldsymbol{x}}}
\newcommand{\Bf}{{\boldsymbol{f}}}
\newcommand{\Bw}{{\boldsymbol{w}}}
\newcommand{\By}{{\boldsymbol{y}}}

\newcommand{\OR}{{\mathscr{O}}_{R}}
\newcommand{\DR}{{\mathcal{D}}_{R}}

\newcommand{\BP}{{\boldsymbol{\Psi}}}

\begin{document}
	
 \title[Uniqueness and Asymptotic behavior]
	{Uniqueness and asymptotic behavior of the stationary Navier-Stokes equations in a slab}

     \author{Han Li}
\address{School of Mathematical Sciences, Soochow University, No. 1 Shizi Street, Suzhou, China}
\email{hli3@stu.suda.edu.cn}

		\author{Jingwen Han}
		\address{School of Mathematics and Statistics, Anhui Normal University, Wuhu, China}
		\email{mathjwh95@ahnu.edu.cn}

	\begin{abstract}
In this paper, we investigate the uniqueness and asymptotic behavior of the stationary Navier-Stokes equations in a slab domain with no-slip boundary conditions.    Specifically, under the given external  force and the  small Poiseuille flow assumptions, we prove that there is a  $H^{1}$ solution, and obtain the pointwise  decay around the Poiseuille flow at far field. Furthermore, if the  force is small, the solution is shown to be unique.  The key point of the proof is the estimate of the Dirichlet integral of the solutions in the truncated domains  and the Stokes regularity estimates.

        \end{abstract}


	
	\maketitle

	\section{Introduction and Main Results}\label{sec1}
The stationary Navier-Stokes equations are of fundamental importance in fluid mechanics, describing the steady motion of viscous incompressible fluids. In this paper, we are interested in the uniqueness and  pointwise asymptotic behavior of the solutions to the nonhomogeneous stationary Navier-Stokes equations in a three-dimensional slab domain $\Omega = \mathbb{R}^2 \times (0,1)$, given by:
\begin{equation}\label{eqsteadynfsf}
\begin{cases}
-\Delta \Bu + (\Bu \cdot \nabla)\Bu + \nabla P = \Bf, \\ 
\nabla \cdot \Bu = 0,
\end{cases} \quad \text{in } \Omega.
\end{equation}
Here, the unknown vector field $\Bu = (u^1, u^2, u^3)$ represents the fluid velocity, $P$ is the pressure, and $\Bf = (f^1, f^2, f^3)$ is the external force. The slab domain is supplemented with the  no-slip boundary conditions:
\begin{equation}\label{noslipboun}
\Bu = 0, \quad \text{at } x_3 = 0 \text{ and } 1.
\end{equation}

The solvability and asymptotic properties of the Stokes and Navier-Stokes equations in domains with noncompact boundaries---such as cylindrical, quasi-cylindrical, and layer-like outlets---have been the subject of extensive mathematical research over the past few decades.   Fundamental results regarding the weak solvability of    Stokes and Navier-Stokes equations in such general unbounded domains were established by Heywood, Ladyzhenskaya, Solonnikov,  Pileckas, and Nazarov (see,  \cite{LSZNSL80,Heywood1976,PileJSM84,Ladyzhenskaya1977,Solonnikov1981,Solonpile1977}). In \cite{Pimmmas96,PileRs9,Pimtfm96}, the classical solvability and asymptotic behavior of solutions to these problems were proved. 


For the homogeneous case ($\Bf=0$), a Liouville-type rigidity theorem, recently established by Bang et al. \cite{aBGWX} and Carrillo et al. \cite{BPZZ}, asserts that any D-solution---namely, a weak solution with finite Dirichlet integral $\int_{\Omega} |\nabla \Bu|^2 \dd \Bx < \infty$---in a slab with no-slip boundary conditions must identically vanish.
 This rigidity  implies that any physically meaningful non-trivial flow in a slab, such as the well-known Poiseuille flow,  possesses an infinite Dirichlet integral.  See also Tsai \cite{TT21} for some extensions of Liouville-type theorem  under the integrability conditions.

 The pioneering investigations of solutions in unbounded domains with unbounded Dirichlet integrals date back to Ladyzhenskaya and Solonnikov \cite{LSZNSL80}. Subsequently, a series of comprehensive works  for domains that are layer-like outlets at infinity were established by  Nazarov, Pileckas, and Specovius-Neugebauer\cite{NP99JMFM,1NP99JMFM,PMSLI2002,NPZAA01,PMZAAS08,KPSNAA10}, where they obtained  systematically explored the solvability and asymptotic properties of weak solutions in various weighted and anisotropic functional spaces. 

In particular, Nazarov and Pileckas~\cite{NP99JMFM} investigated the solvability of the weak solutions to the Navier-Stokes equations in a domain that coincides with the three-dimensional layer $\mathbb{R}^2 \times (0,1)$ outside a compact ball. By employing the Banach contraction mapping principle, they constructed non-trivial solutions that asymptotically approach classical laminar flows, such as the plane-parallel Poiseuille flow, Couette flow, and rotational flow. Specifically, for the Poiseuille flow at far field, they proved that if the given external force and the  far-field mass flux are sufficiently small, there exists a unique weak solution whose perturbation field possesses a finite Dirichlet integral. 
In this paper, we try to investigate the problem with large external force. 

The classical Leray problem for the stationary Navier–Stokes equations in channels and pipes has motivated lots of works on the asymptotic behavior of solutions in domains with noncompact boundaries. For two-dimensional channels with no-slip boundary conditions and prescribed flux $\Phi$, the first author and sha \cite{LSJDE25}
proved existence of weak solutions for any flux $\Phi$,
 as long as the cross-section width grows sublinearly at infinity, uniqueness for sufficiently small flux $\Phi$
under additional geometric assumptions. They also obtained the pointwise decay rate of the velocity for arbitrary flux.  In a complementary three-dimensional direction, for layer-like domains that coincide outside a ball with the  slab $\Omega = \mathbb{R}^2 \times (0,1)$,  Pileckas \cite{PMSLI2002} established that weak solutions with nonzero flux behave at infinity like solutions of the linear Stokes system under the certain weighted integral assumptions. This analysis was refined by Pileckas and Specovius-Neugebauer \cite{KPSNAA10}, who showed that if the data decay sufficiently fast, then the first three asymptotic terms of any solution growing not too fast at infinity have the same structure as those of the linear Stokes problem.  In this paper, we also study the pointwise asymptotic behavior of solutions in the slab domain.

Precisely, we are interested in steady fluid motions in a perfect slab $\Omega = \mathbb{R}^2 \times (0,1)$ that are perturbed by a large localized external force,  and converge to a Poiseuille flow at spatial infinity. To  formulate the problem from a clear physical perspective, we introduce  the concept of the line-flux.

Following the geometric construction in \cite[Section 2.3]{NP99JMFM}, let $\Be^{(\alpha)} = (\cos\alpha, \sin\alpha, 0)$ denote a prescribed horizontal direction. We introduce a rectangular cross-section perpendicular to $\Be^{(\alpha)}$ with a finite lateral width $l > 0$, defined by
\begin{equation}
    Q_l^{(\alpha)} = \left\{ \Bx \in \Omega : \eta_1 = \text{const.} \in \mathbb{R}, \; |\eta_2 - \eta_{20}| < \frac{l}{2}, \; \eta_{20} \in \mathbb{R}, \; x_3 \in (0,1) \right\},
\end{equation}
where $\eta_1 = x_1\cos\alpha + x_2\sin\alpha$ coincides with the flow direction, and $\eta_2 = -x_1\sin\alpha + x_2\cos\alpha$ is the transverse coordinate. Then, the average line-flux $F$ across $Q_l^{(\alpha)}$ is rigorously defined by the limit:
\begin{equation}\label{eq:flux_u}
    F = \lim_{l \to \infty} \frac{1}{l} \int_{Q_l^{(\alpha)}} \Bu(x) \cdot \Be^{(\alpha)} \, \dd\eta_2 \dd x_3.
\end{equation}

We prescribe a constant scalar far-field line-flux $F \in \mathbb{R}$ across the transverse cross-section perpendicular to $\Be^{(\alpha)}$. According to the mass conservation in the slab domain, this prescribed flux $F$ uniquely determines an exact laminar Poiseuille background flow $\bBW(x_3)$ directed along $\Be^{(\alpha)}$. Since $\bBW$ depends solely on the vertical coordinate $x_3$, the transverse integration over $\eta_2$ and the averaging factor $1/l$ naturally cancel out. Thus, its corresponding line-flux strictly matches $F$:
\begin{equation}\label{eq:flux_W}
    \bBW(x_3) = 6F x_3(1-x_3)\Be^{(\alpha)}, \quad \text{with} \quad \mathcal{F}_W = \int_0^1 \bBW(x_3) \cdot \Be^{(\alpha)} \, \dd x_3 = F.
\end{equation}

By utilizing this explicit correspondence, we show there is a velocity field $\Bu = \bBW + \Bv$ whose far-field line-flux  approaches to the prescribed constant $F$ along the direction $\Be^{(\alpha)}$.

Now we state the main results of this paper.

\begin{theorem}\label{DS}
Let the slab domain $\Omega=\mathbb{R}^{2}\times(0,1)$ and let $V'$ denote the dual space of $V = D_{0,\sigma}^{1,2}(\Omega)$.  Assume that the external force $\Bf \in V^{\prime}$.  Let    $F \in \R$  be a prescribed far-field constant line-flux, and $\bBW(x_3)$ be the  corresponding Poiseuille flow defined in \eqref{eq:flux_W}. 

There exists a small constant $\epsilon_0$ such that if the prescribed far-field flux satisfies  $|F |\le \epsilon_0$, then there exists a weak solution $\Bu = \bBW + \Bv$ to the stationary Navier-Stokes equations \eqref{eqsteadynfsf} satisfying the no-slip boundary conditions \eqref{noslipboun}, such that the perturbation field $\Bv$ belongs to $H^{1}(\Omega)$.

Furthermore,  if the external force $\Bf$ has a compact support,  the solution satisfies the following properties:
\begin{itemize}
    \item[(i)] Pointwise decay and flux convergence:  There exists a  constant $\alpha>0$ such that for any sufficiently large $R_{1}$, the following pointwise decay holds in the far field:
\begin{equation*}\label{eq:pointwise_decay_thm}
|\Bu(x) - \bBW(x_3)| \le C R^{-\alpha}, \quad \forall \Bx = (x_1, x_2, x_3) \in \Omega \text{ with } |x'| \ge R_{1},
\end{equation*}
where $x' = (x_1, x_2)$, $R = |x'|$, and the  constant $C > 0$ and $\alpha$ depend only on $\Omega$ and $F$.

    Consequently, the line-flux of $\boldsymbol{u}$ asymptotically converges to the prescribed constant $F$ along the direction $\boldsymbol{e}^{(\alpha)}$ at spatial infinity, i.e.,
    \begin{equation}\label{eq:flux_limit}
    \lim_{|\boldsymbol{x}'| \to \infty} \int_0^1 \boldsymbol{u}(\boldsymbol{x}', x_3) \cdot \boldsymbol{e}^{(\alpha)} \, \dd x_3 = F.
    \end{equation}
      \item[(ii)] Uniqueness: In addition, if there exists a small constant $\epsilon_1$ such that  the external force satisfies  $\|\Bf\|_{V'(\Omega)} \le \epsilon_1$, then the solution $\Bu$ obtained above is unique in the class of functions satisfying $\Bu - \bBW \in H^{1}(\Omega)$.
\end{itemize}
\end{theorem}

\begin{remark}
    The assumption that $\Bf$ has a compact support can be replaced by  $\Bf$ has the decay property at infinity.
\end{remark}

\begin{remark}\label{rem:layer_like}
Although the analysis in this paper is formulated within a  perfect slab $\Omega = \mathbb{R}^2 \times (0,1)$, our main results can be directly extended to the class of layer-like domains investigated by Nazarov and Pileckas \cite{NP99JMFM}, which coincide with the  slab  outside a compact set (e.g., a sufficiently large ball).
\end{remark}

\begin{remark}\label{rem:flux_mechanism}
The flux convergence \eqref{eq:flux_limit} is a  consequence of Proposition \ref{prop4.1}. Since the perturbation field $\Bv$ belongs to  $ D_{0,\sigma}^{1,2}(\Omega)$, it inherently decays to zero at spatial infinity, causing its corresponding mass flux to vanish as $|x'| \to \infty$. It then follows immediately from the decomposition $\Bu = \bBW + \Bv$ that the total far-field line-flux of $\Bu$ approaches to the flux of the Poiseuille flow $\bBW$, which is precisely $F$.
\end{remark}

\begin{remark}\label{rem:np_comparison}
Under the same prescribed far-field flux setting, Theorem \ref{DS} has two improvements compared to the previous results in \cite{NP99JMFM}. Firstly, we optimize the solvability condition by  removing the smallness restriction on the external force $f$ for the existence of weak solutions. Secondly, beyond the global energy bounds, we further establish the pointwise algebraic decay rate $|\Bu(x) - \bBW(x_3)| \le C |x'|^{-\alpha}$ in the far field.
\end{remark}


    

	Now we outline  the idea of the proof for the main results.  In order to prove the asymptotic  behavior of the stationary Navier-Stokes equations around the Poiseuille flow, firstly, we use the Bogovskii map, the structure of the Navier-Stokes equations to establish the Saint-Venant type estimate for the Dirichlet integral of $\Bv$ over the finite subdomain.  Then we utilize the Stokes regularity estimates to obtain the pointwise decay. The Saint-Venant principle was initially used to study the solutions of elastic equations \cite{RATARMA65,JKKTARMA66}. The idea was generalized in \cite{LSZNSL80} to investigate the uniqueness and asymptotic behavior of solutions to the Navier-Stokes equations in pipe domains. Recently, it was applied to study the Liouville-type theorems for the Navier-Stokes equations in a slab \cite{aBGWX}.
	

The rest of this paper is organized as follows. In Section \ref{Sec2}, we introduce the Bogovskii map, some notations, and inequalities, which are used in this paper. Section \ref{Sec3} is devoted to establishing the estimate of Dirichlet integrals on  truncated domains, and  the pointwise decay rate of the solutions.  The uniqueness of the Navier-Stokes equations in a slab with a small external force  are proved in Section \ref{Sec4}.

	\section{PRELIMINARIES}\label{Sec2} 
	Some  notations are given below. 
	We denote the cut-off domain $\mathcal{D}_R =(R-1, R)\times(0, 2\pi)\times(0, 1)$, $\OR=(B_{R}\setminus \overline{B_{R-1}})\times (0,1)$,  where $B_{R}=\left\{(x_{1},x_{2})\in\mathbb{R}^{2}: x^{2}_{1}+x^{2}_{2}<R^2\right\}$. For any $\Bx\in \mathbb{R}^3$, define $\mathscr{B}_r(\Bx)=\{\By\in \mathbb{R}^3: |\By-\Bx|<r\}$.
	Define the space $D_{0}^{1,2}(\Omega)$ be the closure of $C^{\infty}_{c}(\Omega)$ in the norm $\|\Bu\|_{D^{1,2}(\Omega)}=\|\nabla\Bu\|_{L^{2}(\Omega)}$, and $V=D_{0,\sigma}^{1,2}(\Omega)$ denotes the subspace of $D_{0}^{1,2}(\Omega)$ consisting of divergence-free functions.    $V^{\prime}$ is the dual space of $V$. $C_{c,\sigma}^{\infty}{(\Omega)}$ is the   subspace of $C_{c}^{\infty}{(\Omega)}$  consisting of divergence-free functions.
    
 The standard cylindrical coordinates $(r,\theta,  z)$ in
	$\mathbb{R}^3$  are defined as follows:
	\[
	\Bx=(x_{1}, x_{2}, x_{3})=(r\cos\theta, r\sin\theta, z).
	\]
	In cylindrical coordinates, the velocity $\Bu$  can be written as 
	\begin{equation}\label{cylindcoo}
		\Bu=u^{r}(r,\theta,z)\Be_{r}+u^{\theta}(r,\theta,z)\Be_{\theta}+u^{z}(r,\theta,z)\Be_{z},
	\end{equation}
	where scalar components $u^{r}, u^{\theta}, u^{z}$ are called radial, swirl and axial velocity, respectively, and the basis vectors $\Be_{r}, \Be_{\theta}, \Be_{z}$ are
	\[
	\Be_{r}=(\cos\theta, \sin\theta, 0), \quad
	\Be_{\theta}=(-\sin\theta, \cos\theta,0), \quad
	\Be_{z}=(0,0,1).
	\]
  Cylindrical coordinates will be utilized in the following proof.  
    
	\begin{definition}
		Let $1\leq p \leq \infty$ and  $D $ be a bounded domain in $\mathbb{R}^n$. Denote
		\[
		L^p_0(D)=\left\{v(\Bx):v\in L^p(D), \int_D v(\Bx)\,\dd \Bx=0\right\}. 
		\]
	\end{definition}
	
	Then we introduce the Bogovskii map, which gives a solution to the divergence equation.  	The following is from   \cite[Lemma 2.1]{aBGWX}.
    Its general form can be refer to  Bogovskii\cite{BM}, see \cite[Section III.3]{GAGP11} and \cite[Section 2.8]{TT18}.
	\begin{lemma}\label{Bogovskii}
		Let $D$ be a bounded Lipschitz domain in $\mathbb{R}^n$, $n\geq 2 $.  For any $p\in (1, +\infty)$, there is a linear map $\boldsymbol{\Phi}$ that maps a scalar function $g\in L^p_0(D)$ to a vector field $\bBV = \boldsymbol{\Phi} g \in W_0^{1, p}(D; \mathbb{R}^n)$ satisfying
		\be \nonumber
		{\rm div}~\bBV = g \ \text{in}\,\, D \quad \text{and} \quad \|\bBV\|_{W_{0}^{1, p}(D)} \leq C (D, p) \|g\|_{L^p(D)}.
		\ee
		
		In particular,
      for any $g \in L_0^2(\mathcal{D}_R)$,
			the   vector valued function $\bBV = \boldsymbol{\Phi} g \in W_0^{1,2}( \mathcal{D}_R; \mathbb{R}^3)$ satisfies
			\be \nonumber
			\partial_r V^r + \partial_\theta V^\theta +  \partial_z V^z =g \ \   \mbox{in}\,\, \mathcal{D}_R
			\quad
			\text{and}
			\quad
			\|\overline{\nabla } \bBV\|_{L^2(\mathcal{D}_R)}
			\leq C \|g\|_{L^2( \mathcal{D}_R)},
			\ee
			where $\overline{\nabla} = (\partial_r, \ \partial_\theta, \ \partial_z ) $ and $C$ is a constant independent of $R$.
		\end{lemma}

	The Gagliardo-Nirenberg interpolation inequality in bounded domains is as follows.
	\begin{lemma}\label{G-Ninequality}(\hspace{1sp}\cite[Theorem 1.2]{LZ})
		Let $D$ be a bounded Lipschitz domain. Assume that $1 \le q,r \le + \infty$, $k,j \in \mathbb{N}$ with $j <k$, $\frac{j}{k}\le \theta\le 1$ and $p \in \mathbb{R}$ such that 
		\[
		\frac{1}{p}=\frac{j}{n}+\theta\left( \frac{1}{r}-\frac{k}{n}\right)+(1-\theta)\frac{1}{q}.
		\]
		Then there exists a constant C depending only on $n,k,q,r,\theta$ and $D$ such that for any $\Bu \in W^{k,r}(D)\cap L^q(D) $,
		\begin{equation}\label{GN}
			\begin{aligned}
				\|\nabla^j \Bu\|_{L^{p}(D)} \le C\|\nabla^k \Bu\|_{L^r(D)}^{\theta}\|\Bu\|_{L^q(D)}^{1-\theta}+C\|\Bu\|_{L^q(D)},
			\end{aligned}
		\end{equation}
		with the exception that if $1<r<+\infty $ and $k-j-\frac{n}{r} \in \mathbb{N}$, we must take $\frac{j}{k}\le \theta<1 $.
		
	\end{lemma}

    
Next, we will give the following lemma on the interior regularity of the Stokes equations, the proof can be found in \cite[Theorem IV. 4.1]{GAGP11}.
\begin{lemma}\label{interior regularity}
	Assume that $\Omega$ is an arbitrary domain in $\mathbb{R}^n$ with $n\ge 2$. Let $\Bv$ be weakly divergence-free with $\nabla\Bv\in L^q_{loc}{(\Omega)}$,     $1<q<\infty$,  and satisfying 
	\[\int_{\Omega}\nabla\Bv:\nabla\boldsymbol{\varphi}\,\dd \Bx=\int_{\Omega} \Bf\cdot  \boldsymbol{\varphi}  \,\dd \Bx  ~~\text{ for any }\boldsymbol{\varphi}\in C_{c,\sigma}^{\infty}{(\Omega)}.\]
	
	If $\Bf\in W_{loc}^{m,q}{(\Omega)}$ for some $ m\ge 0$, then it follows that $\Bv \in W_{loc}^{m+2,q}{(\Omega)}$, $ p \in W_{loc}^{m+1,q}{(\Omega)}$, where $p$ is the pressure associated to $\Bv$. Further the following inequality holds:
\begin{equation}\label{6-3}
	\|\nabla^{m+2}\Bv\|_{L^{q}{(\Omega')}}+\|\nabla^{m+1}p\|_{L^{q}{(\Omega')}}\le C\left(\|\Bf\|_{W^{m,q}{(\Omega'')}}+\|\Bv\|_{W^{1,q}{(\Omega''\setminus\Omega')}}+\|p\|_{L^{q}{(\Omega''\setminus\Omega')}}\right),
\end{equation}
where $\Omega'$, $\Omega''$ are arbitrary bounded subdomains of $\Omega$ with $\overline{\Omega'}\subset\Omega''$, $\overline{\Omega''}\subset \Omega$,
and $C=C(n,q,m,\Omega',\Omega'')$.
\end{lemma}

\begin{remark}\label{qe}
If the domain $\Omega''\setminus\Omega'$, in the previous lemma, satisfies the cone condition, we can remove the term involving the pressure on the right-hand side of \eqref{6-3} by modifying $p$ with a constant. Therefore, we obtain 
\begin{equation}\label{6-5}
	\begin{aligned}
	\|\nabla^{m+2}\Bv\|_{L^{q}{(\Omega')}}+ \|\nabla^{m+1}p\|_{L^{q}{(\Omega')}}\le C\left(\|\Bf\|_{W^{m,q}{(\Omega'')}}+\|\Bv\|_{W^{1,q}{(\Omega''\setminus\Omega')}}\right).
	\end{aligned}
\end{equation}
\end{remark}
  Finally, we will give the following regularity estimates of the solutions to the Stokes equations in the half space  $\mathbb{R}^{3}_{+}$,   which will be used in deal with the decay rate near the boundary. The proof can be found in \cite[Theorem IV. 3.2]{GAGP11}.
	\begin{lemma}\label{bpm}
Assume that $m\ge 0$ and $1<q<\infty$. For every 
	\[
	\Bf\in W^{m,q}{(\R_+^n)}\text{ and } g \in W^{m+1,q}(\R^n_+),
	\]
 there exists a pair of functions $(\Bv, p)$ such that
\[
\Bv \in W^{m+2,q}{(Q)}, \quad p\in W^{m+1,q}{(Q)},
\]
for all open cubes $Q \subset \R_+^n$, solving  the following non-homogeneous Stokes equations
\begin{equation}
	 \left\{\begin{aligned}
	&-\Delta \Bv+\nabla p= \Bf&\text{ in } \R^n_+,\\
	&\nabla \cdot \Bv = g&\text{ in } \R^n_+,\\
	&\Bv= 0& \text{ on } \partial \R^n_+.
	\end{aligned}\right.
\end{equation} 
Moreover, for all $l\in[0,m]$, we have
\begin{equation}\label{bmpeq}
	\begin{aligned}
		\|\nabla^{l+2}\Bv\|_{L^{q}{(\R^n_+)}}+\|\nabla^{l+1}p\|_{L^{q}{(\R^n_+)}}\le C\left( \|\nabla^{l}\Bf\|_{L^{q}{(\R^n_+)}}+ \|\nabla^{l+1}g\|_{L^{q}{(\R^n_+)}} \right),
	\end{aligned}
\end{equation}
where $C=C(n,q,m)$. 
\end{lemma}


\section{Existence and Asymptotic behavior for solutions in a slab}\label{Sec3}

 This section is devoted to establishing the existence of the finite-energy perturbation $\Bv$ and investigating the pointwise asymptotic behavior around Poiseuille flow  $ \bBW(x_3) = 6F x_3(1-x_3)\Be^{(\alpha)}$  at spatial infinity. By subtracting the homogeneous system for $\bBW$ from the non-homogeneous equations for $\Bu$, the perturbation field $\Bv$ and the associated  pressure $p = P - P_{\bBW}$ satisfy the following  perturbation system:
\begin{equation}\label{eq:pert_v_global}
\begin{cases}
-\Delta \Bv + (\Bv\cdot\nabla)\Bv + (\bBW\cdot\nabla)\Bv + (\Bv\cdot\nabla)\bBW + \nabla p = \Bf, & \text{in } \Omega, \\
\nabla\cdot \Bv = 0, & \text{in } \Omega,
\end{cases}
\end{equation}
supplemented with the  no-slip boundary conditions $\Bv=0$ at $x_3=0, 1$.

\begin{pro}\label{prop4.1}
Assume that the  external force $\Bf\in V^{\prime}$ and  the flux of Poiseuille flow satisfies   $|F| \le \epsilon_0$.   Then there exists a generalized weak solution  $\Bu =\bBW + \Bv$ to the Navier-Stokes equations \eqref{eqsteadynfsf}, where  the perturbation field $\Bv \in H^1(\Omega)$.

Furthermore, if $\Bf$  has a compact support, then there exist  constants $C>0$ and $\alpha>0$ such that for any sufficiently large $R>R_{1}$, the following far-field energy decay estimate holds:
\begin{equation}\label{eq:prop4.1_decay}
\|\Bu-\bBW\|_{H^{1}(\Omega_{>R})}\le CR^{-\alpha},
\end{equation}
where $\Omega_{>R}=\{(x_{1},x_{2})\in\mathbb{R}^{2}:|x^{\prime}|>R\}\times(0,1)$ denotes the exterior domain, and the constants $C$ and $\alpha$ depend only on $\Omega$ and $F$.
\end{pro}

\begin{proof}[Proof for Proposition  \ref{prop4.1}]
The proof is divided into two steps.

\emph{Step 1.} \emph{Existence and a priori bounds.}
We construct the solution to the perturbation system \eqref{eq:pert_v_global} by employing the standard expanding domain method (or Galerkin approximation) in $H^1(\Omega)$. Multiplying the momentum equation in \eqref{eq:pert_v_global} by $\Bv$ and integrating over the  slab $\Omega$, the non-linear convection terms vanish identically due to the divergence-free constraints $\nabla \cdot \Bv = 0$, $\nabla \cdot \bBW = 0$, and the no-slip boundary conditions on $x_3=0,1$. Specifically, we have $\int_{\Omega} (\Bv\cdot\nabla)\Bv \cdot \Bv \, \dd \Bx = 0$ and $\int_{\Omega} (\bBW\cdot\nabla)\Bv \cdot \Bv \, \dd \Bx = 0$. This leads to the global energy identity:
\begin{equation}\label{eq:global_energy_id}
\int_{\Omega} |\nabla \Bv|^2 \, \dd \Bx+ \int_{\Omega} (\Bv\cdot\nabla)\bBW \cdot \Bv  \,\dd \Bx = \int_{\Omega} \Bf \cdot \Bv \, \dd \Bx.
\end{equation}
By applying Hölder's inequality and the Poincaré inequality $\|\Bv\|_{L^2(\Omega)} \le C \|\nabla \Bv\|_{L^2(\Omega)}$, one has
\begin{equation}
\left| \int_{\Omega} (\Bv\cdot\nabla)\bBW \cdot \Bv \, \dd \Bx\right| \le \|\nabla \bBW\|_{L^\infty(\Omega)} \|\Bv\|_{L^2(\Omega)}^2 \le C\|\nabla \bBW\|_{L^\infty(\Omega)} \|\nabla \Bv\|_{L^2(\Omega)}^2.
\end{equation}
Since the flux of Poiseuille flow is small ($|F| \le \epsilon_0$), let  $C \|\nabla \bBW\|_{L^\infty(\Omega)} \le \frac{1}{2}$,  and \eqref{eq:global_energy_id} reduces to:
\begin{equation}
\frac{1}{2} \|\nabla \Bv\|_{L^2(\Omega)}^2 \le \|\Bf\|_{V'} \|\nabla \Bv\|_{L^2(\Omega)} \implies \|\nabla \Bv\|_{L^2(\Omega)} \le 2\|\Bf\|_{V'}.
\end{equation}
Since $\Bf\in V^{\prime}$, this a priori bound ensures the  existence of a weak solution $\Bv \in H^1(\Omega)$. Here one may refer to \cite[Section 2.4]{TT18} or  \cite[Section 5.1]{Ladyzhenskbook} for the expanding domain method.


\emph{Step 2.} \emph{$H^{1}$ decay rate.}
Since the external force $\Bf$ has a compact support in $\Omega$, there exists a constant $R_0 > 0$ such that $\Bf(x) \equiv 0$ for all $|x'| \ge R_0$. For any $R \ge \max\{4, R_0 + 2\}$, we introduce the  cut-off function:
\begin{equation}
\varphi_{R}(r)=\begin{cases}0,&r<R-1,\\ r-R+1,&R-1\le r\le R,\\ 1,&r>R.\end{cases}
\end{equation}
Multiplying the perturbation equation \eqref{eq:pert_v_global} by $\varphi_{R}(r)\Bv$ and integrating by parts over $\Omega$, we obtain
\begin{equation}\label{eq:sv_estimate}
\begin{aligned}
&\int_{\Omega}|\nabla \Bv|^{2}\varphi_{R}\,\dd \Bx+\int_{\Omega}\nabla\varphi_{R}\cdot\nabla \Bv\cdot \Bv\,\dd \Bx-\int_{\Omega}\frac{1}{2}|\Bv|^{2}\Bv\cdot\nabla\varphi_{R}\,\dd \Bx+\int_{\Omega}(\Bv\cdot\nabla)\bBW\cdot(\Bv\varphi_{R})\,\dd \Bx \\
&-\int_{\Omega}\frac{1}{2}|\Bv|^{2}\bBW\cdot\nabla\varphi_{R}\,\dd \Bx-\int_{\Omega}p \Bv\cdot\nabla\varphi_{R}\,\dd \Bx = \int_{\Omega} \Bf \cdot (\Bv \varphi_{R}) \,\dd \Bx.
\end{aligned}
\end{equation}
Since $R-1 > R_0$, the support of the cut-off function $\varphi_{R}$ is strictly disjoint from the compact support of $\Bf$. Consequently, the term involving the external force on the right-hand side identically vanishes: $\int_{\Omega} \Bf \cdot (\Bv \varphi_{R})\, \dd \Bx = 0$.

  Due to  the Poincar\'e inequality, it holds that
  \begin{equation}
  \begin{split}
     \left|\int_{\Omega}(\Bv\cdot\nabla)\bBW\cdot(\varphi_R\Bv)\,\dd \Bx\right|
      &\leq\|\nabla\bBW\|_{L^{\infty}(\Omega)}\int_{\Omega}|\Bv|^{2}\varphi_R\,\dd \Bx\\
      &\leq \|\nabla\bBW\|_{L^{\infty}(\Omega)}\int_{\Omega}|\partial_{z}(\Bv\sqrt{\varphi_R})|^{2}\,\dd \Bx\\
      &\leq \|\nabla\bBW\|_{L^{\infty}(\Omega)} \int_{\Omega}|\nabla\Bv|^{2}\varphi_R\,\dd \Bx.
       \end{split}
  \end{equation}
  This, together with \eqref{eq:sv_estimate} and $\|\nabla\bBW\|_{L^{\infty}(\Omega)}\ll 1$ (since 
 $|F| \le \epsilon_0$) yields
  \begin{equation}\label{initieqf}
 \begin{split}
\int_{\Omega}|\nabla\Bv|^{2}\varphi_R\,\dd \Bx\leq& \left| \int_{\Omega}\nabla\varphi_R \cdot\nabla\Bv\cdot\Bv\,\dd \Bx\right|+\left|\frac{1}{2}\int_{\Omega}|\Bv|^{2}\Bv\cdot\nabla\varphi_R\,\dd \Bx \right|+\left|\frac{1}{2}\int_{\Omega}|\Bv|^{2}\bBW\cdot\nabla\varphi_R\,\dd \Bx \right|\\
&+\left|\int_{\Omega}p\Bv\cdot\nabla\varphi_R\,\dd\Bx\right|.
\end{split}
  \end{equation}
Note that by the Poincar\'e inequality, we have
\begin{equation}\label{estima1}
    \left| \int_{\Omega}\nabla\varphi_R \cdot\nabla\Bv\cdot\Bv\,\dd \Bx\right|\leq C\|\nabla\Bv\|_{L^{2}(\OR)}\|\Bv\|_{L^{2}(\OR)}
    \leq C\|\nabla\Bv\|^{2}_{L^{2}(\OR)}
\end{equation}
and 
\begin{equation}\label{estima2}
    \left|  \frac{1}{2}\int_{\Omega}|\Bv|^{2}\bBW\cdot\nabla\varphi_R\,\dd \Bx  \right|\leq C\|\bBW\|_{L^{\infty}(\OR)} \|\nabla\Bv\|^{2}_{L^{2}(\OR)}.
\end{equation}
    Denote $\overline{\nabla}(v^{r},v^{\theta},v^{z})=(\partial_{r},\partial_{\theta},\partial_{z})(v^{r},v^{\theta},v^{z})$, since
   \begin{equation}\label{prvrpzv}
|\partial_{r}v^{r}|+|\partial_{r}v^{\theta}|+|\partial_{r}v^{z}|+|r^{-1}\partial_{\theta}v^{z}|+|\partial_{z}v^{r}|+|\partial_{z}v^{\theta}|+|\partial_{z}v^{z}|\leq C|\nabla\Bv|
\end{equation}
and
 \begin{equation}\label{pthetav}
|r^{-1}\partial_{\theta}v^{r}-r^{-1}v^{\theta}|+|r^{-1}\partial_{\theta}v^{\theta}+r^{-1}v^{r}|\leq C|\nabla\Bv|,
\end{equation}
it holds that
 \begin{equation}\label{pthethavt}
|r^{-1}\partial_{\theta}v^{r}|+|r^{-1}\partial_{\theta}v^{\theta}|\leq C|\nabla\Bv|+|r^{-1}v^{\theta}|+|r^{-1}v^{r}|.
\end{equation}
By the Gagliardo-Nirenberg inequality and the Poincar\'e inequality, \eqref{prvrpzv} and \eqref{pthethavt}, one has
\begin{equation}\label{estima3}
\begin{split}
    \left|\frac{1}{2}\int_{\Omega}|\Bv|^{2}\Bv\cdot\nabla\varphi_R\,\dd \Bx \right|&\leq CR\int_{0}^{1}\int_{0}^{2\pi}\int_{R-1}^{R}|\Bv|^{3}\,\dd r\dd \theta\dd z\\
    &\leq CR\|(v^{r},v^{\theta},v^{z})\|^{\frac{3}{2}}_{L^{2}(\DR)}\|\overline{\nabla}(v^{r},v^{\theta},v^{z})\|^{\frac{3}{2}}_{L^{2}(\DR)}\\
    &\leq CR\cdot R^{-\frac{3}{4}}\|\Bv\|^{\frac{3}{2}}_{L^{2}(\OR)}
    \left( R^{\frac{1}{2}}\|\nabla\Bv\|_{L^{2}(\OR)}+R^{-\frac{1}{2}} \|\Bv\|_{L^{2}(\OR)}     \right)^{\frac{3}{2}} \\
    &\leq CR\|\nabla\Bv\|^{3}_{L^{2}(\OR)}.
 \end{split}
\end{equation}

Note that
\begin{equation}\label{PrePin}
\left|\int_{\Omega}p\Bv\cdot\nabla\varphi_R\,\dd\Bx\right|=\left|\int_{0}^{1}\int_{0}^{2\pi}\int_{R-1}^{R}pv^{r}r\,\dd r \dd \theta \dd z\right|.
\end{equation}
The  divergence free condition of the three-dimensional Navier-Stokes equations in cylindrical coordinates is
\begin{equation}
    \partial_{r}v^{r}+\frac{v^{r}}{r}+\frac{ \partial_{\theta}v^{\theta}}{r}+ \partial_{z}v^{z}=0.
\end{equation}
Since the no-slip boundary conditions \eqref{noslipboun}, for every fixed $r\geq 0$, one has
\begin{equation}
    \partial_{r}\int_{0}^{1}\int_{0}^{2\pi}
   rv^{r}\,\dd \theta\dd z=-\int_{0}^{1}\int_{0}^{2\pi}\p_{\theta}v^{\theta}+\p_{z}(rv^{z})\,\dd \theta\dd z=0.
\end{equation}
Then it holds that
\begin{equation}
    \int_{0}^{1}\int_{0}^{2\pi}
   rv^{r}\,\dd \theta\dd z=0 \ \ \  \text{and} \ \ \  \int_{0}^{1}\int_{0}^{2\pi}\int_{R-1}^{R}
   rv^{r}\,\dd r \dd \theta\dd z=0.
\end{equation}
By Lemma\,\ref{Bogovskii}, there is a vector valued function $\BP_{R}(r,\theta,z)\in H^{1}_{0}(\DR;\mathbb{R}^{3})$ such that
\begin{equation}
\p_{r}\Psi_{R}^{r}+\p_{\theta}\Psi_{R}^{\theta}+\p_{z}\Psi_{R}^{z}=rv^{r} \ \ \ \text{in}\  \DR
\end{equation}
and 
\begin{equation}\label{eqPrPsies}
    \|\p_{r}\BP_{R}\|_{L^{2}(\DR)}+ \|\p_{\theta}\BP_{R}\|_{L^{2}(\DR)}+
     \|\p_{z}\BP_{R}\|_{L^{2}(\DR)}\leq C\|rv^{r}\|_{L^{2}(\DR)}\leq CR^{\frac{1}{2}}\|v^{r}\|_{L^{2}(\OR)}.
\end{equation}
Thus, from \eqref{PrePin}, one obtains
\begin{equation}
\begin{split}
\left|\int_{\Omega}p\Bv\cdot\nabla\varphi_R\,\dd\Bx\right|
&=\left|\int_{0}^{1}\int_{0}^{2\pi}\int_{R-1}^{R}p(\p_{r}\Psi_{R}^{r}+\p_{\theta}\Psi_{R}^{\theta}+\p_{z}\Psi_{R}^{z})\,\dd r \dd \theta \dd z\right|\\
&=\left|\int_{0}^{1}\int_{0}^{2\pi}\int_{R-1}^{R}(\p_{r}p\Psi_{R}^{r}+\p_{\theta}p\Psi_{R}^{\theta}+\p_{z}p\Psi_{R}^{z})\,\dd r \dd \theta \dd z\right|.
   \end{split} 
\end{equation}
Furthermore, since $\Bf$ has a compact support,  the perturbation equation \eqref{eq:pert_v_global} with no force in cylindrical coordinates is 
	\begin{equation}\label{eqA8}
		\left\{
		\begin{aligned}
			&-\left(\Delta_{r,\theta, z}- \frac{1}{r^2} \right)v^r+ \frac{2}{r^2} \partial_\theta v^\theta+\left(v^r \partial_r+ \frac{v^\theta}{r}  \partial_\theta+v^z \partial_z\right)v^r - \frac{(v^\theta)^2}{r}+ v^{r}\p_{r} W^{r}+W^{r}\p_{r}v^{r}\\
            &\  -\frac{2v^{\theta}W^{\theta}}{r}+\frac{v^{\theta}\partial_{\theta}W^{r}}{r}+v^{z}\partial_{z}W^{r}+\frac{W^{\theta}\partial_{\theta}v^r}{r}+W^{z}\partial_{z}v^{r} + \partial_{r} p= 0,\\
			&-\left(\Delta_{r,\theta, z}- \frac{1}{r^2}\right) v^\theta - \frac{2}{r^2} \partial_\theta v^r+ \left(v^r \partial_r + \frac{v^\theta }{r}  \partial_\theta + v^z \partial_z\right)v^\theta +
			\frac{v^\theta v^r}{r}+v^{r}\p_{r}W^{\theta}+\frac{v^{\theta}\p_{\theta}W^{\theta}}{r}\\
            &\ +\frac{v^{\theta}W^{r}}{r}+v^{z}\p_{z}W^{\theta}+W^{r}\p_{r}v^{\theta}+\frac{W^{\theta}\p_{\theta}v^{\theta}}{r}+\frac{W^{\theta}v^{r}}{r}+W^{z}\p_{z}v^{\theta} +\frac{1}{r} \partial_{\theta} p
			=0,\\
			&-\Delta_{r,\theta, z}v^z+\left(v^r \partial_r + \frac{v^\theta}{r} \partial_\theta +v^z \partial_z\right)v^z+v^{r}\p_{r}W^{z}+\frac{v^{\theta}\p_{\theta}W^{z}}{r}+v^{z}\p_{z}W^{z}+W^{r}\p_{r}v^{z}\\
           &\ +\frac{W^{\theta}\p_{\theta}v^{z}}{r}+W^{z}\p_{z}v^{z}+ \partial_z p =0,
		\end{aligned}
		\right.
	\end{equation}
    where  $\Delta_{r,\theta,z}=\p^{2}_{r}+\frac{1}{r}\p_{r}+\frac{1}{r^{2}}\p^{2}_{\theta}+\p^{2}_{z}$.
  Integrating by parts yields
    \begin{equation}\label{prEstim}
        \begin{split}
    &\int_{0}^{1}\int_{0}^{2\pi}\int_{R-1}^{R}\p_{r}p\Psi_{R}^{r} \,\dd r \dd \theta \dd z  \\
  =&\int_{0}^{1}\int_{0}^{2\pi}\int_{R-1}^{R}\left(\p^{2}_{r}+\frac{1}{r}\p_{r}+\frac{1}{r^{2}}\p^{2}_{\theta}+\p^{2}_{z}-\frac{1}{r^{2}}\right)v^{r}\Psi_{R}^{r} -\frac{2}{r^{2}}\p_{\theta}v^{\theta}\Psi_{R}^{r}\,\dd r \dd \theta \dd z\\
  &-\int_{0}^{1}\int_{0}^{2\pi}\int_{R-1}^{R}\left(v^r \partial_r+ \frac{v^\theta}{r}  \partial_\theta+v^z \partial_z\right)v^r\Psi_{R}^{r}-\frac{(v^{\theta})^{2}}{r}\Psi_{R}^{r}\,\dd r \dd \theta \dd z\\
  &-\int_{0}^{1}\int_{0}^{2\pi}\int_{R-1}^{R}\left(\p_{r}(v^{r} W^{r})-\frac{2v^{\theta}W^{\theta}}{r}+\frac{v^{\theta}\partial_{\theta}W^{r}}{r}+v^{z}\partial_{z}W^{r}+\frac{W^{\theta}\partial_{\theta}v^r}{r}+W^{z}\partial_{z}v^{r}\right)\Psi_{R}^{r}\,\dd r \dd \theta \dd z\\
  =&-\int_{0}^{1}\int_{0}^{2\pi}\int_{R-1}^{R}(\p_{r}v^{r}\p_{r}\Psi_{R}^{r}+\p_{z}v^{r}\p_{z}\Psi_{R}^{r})+\left( \frac{1}{r^{2}}\p_{\theta}v^{r}-\frac{2}{r^{2}}v^{\theta} \right) \p_{\theta}\Psi_{R}^{r}\,\dd r \dd \theta \dd z\\
  &+\int_{0}^{1}\int_{0}^{2\pi}\int_{R-1}^{R}\left(\frac{1}{r}\p_{r}v^{r}-\frac{1}{r^{2}}v^{r}\right)\Psi_{R}^{r}\,\dd r \dd \theta \dd z+\int_{0}^{1}\int_{0}^{2\pi}\int_{R-1}^{R} \frac{(v^{\theta})^{2}}{r}\Psi_{R}^{r}\,\dd r \dd \theta \dd z   \\
  &-\int_{0}^{1}\int_{0}^{2\pi}\int_{R-1}^{R}\left(v^r \partial_r+ \frac{v^\theta}{r}  \partial_\theta+v^z \partial_z\right)v^r\Psi_{R}^{r}\,\dd r \dd \theta \dd z \\
  &-\int_{0}^{1}\int_{0}^{2\pi}\int_{R-1}^{R}\left(\p_{r}(v^{r} W^{r})-\frac{2v^{\theta}W^{\theta}}{r}+\frac{v^{\theta}\partial_{\theta}W^{r}}{r}+v^{z}\partial_{z}W^{r}+\frac{W^{\theta}\partial_{\theta}v^r}{r}+W^{z}\partial_{z}v^{r}\right)\Psi_{R}^{r}\,\dd r \dd \theta \dd z,
       \end{split}
    \end{equation}
    
  \begin{equation}\label{pthetaes}
        \begin{split}
    &\int_{0}^{1}\int_{0}^{2\pi}\int_{R-1}^{R}\p_{\theta}p\Psi_{R}^{\theta} \,\dd r \dd \theta \dd z  \\
    =&-\int_{0}^{1}\int_{0}^{2\pi}\int_{R-1}^{R}\left[r(\p_{r}v^{\theta}\p_{r} \Psi_{R}^{\theta}+\p_{z}v^{\theta}\p_{z}\Psi_{R}^{\theta})+r^{-1}(\p_{\theta}v^{\theta}\p_{\theta}\Psi_{R}^{\theta}+2v^{r}\p_{\theta}\Psi_{R}^{\theta})\right]\,\dd r \dd \theta \dd z \\
    &-\int_{0}^{1}\int_{0}^{2\pi}\int_{R-1}^{R}r^{-1}v^{\theta}\Psi_{R}^{\theta}\,\dd r \dd \theta \dd z- \int_{0}^{1}\int_{0}^{2\pi}\int_{R-1}^{R}v^{\theta}v^{r}\Psi_{R}^{\theta}\,\dd r \dd \theta \dd z\\
    &-\int_{0}^{1}\int_{0}^{2\pi}\int_{R-1}^{R}r\left(v^r \partial_r + \frac{v^\theta }{r}  \partial_\theta + v^z \partial_z\right)v^\theta\Psi_{R}^{\theta}\,\dd r \dd \theta \dd z-\int_{0}^{1}\int_{0}^{2\pi}\int_{R-1}^{R}rv^{r}\p_{r}W^{\theta}\Psi_{R}^{\theta}\,\dd r \dd \theta \dd z\\
    &-\int_{0}^{1}\int_{0}^{2\pi}\int_{R-1}^{R}(v^{\theta}\p_{\theta}W^{\theta}+W^{r}v^{\theta}+rv^{z}\p_{z}W^{\theta}+rW^{r}\p_{r}v^{\theta}+W^{\theta}\p_{\theta}v^{\theta}+W^{\theta}v^{r}+rW^{z}\p_{z}v^{\theta})\Psi_{R}^{\theta}\,\dd r \dd \theta \dd z
 \end{split}
    \end{equation}
    and
    \begin{equation}\label{prepz}
        \begin{split}
    &\int_{0}^{1}\int_{0}^{2\pi}\int_{R-1}^{R}\p_{z}p\Psi_{R}^{z} \,\dd r \dd \theta \dd z  \\
   =&-\int_{0}^{1}\int_{0}^{2\pi}\int_{R-1}^{R}(\p_{r}v^{z}\p_{r} \Psi_{R}^{z}+\p_{z}v^{z}\p_{z} \Psi_{R}^{z})-(r^{-1}\p_{r}v^{z}\Psi_{R}^{z}-r^{-2}\p_{\theta}v^{z}\p_{\theta} \Psi_{R}^{z})\,\dd r\dd\theta \dd z \\
   &-\int_{0}^{1}\int_{0}^{2\pi}\int_{R-1}^{R}\left(v^r \partial_r+ \frac{v^\theta}{r}  \partial_\theta+v^z \partial_z\right)v^z\Psi_{R}^{z}\,\dd r \dd \theta \dd z-\int_{0}^{1}\int_{0}^{2\pi}\int_{R-1}^{R}v^{r}\p_{r}W^{z}\Psi_{R}^{z}\,\dd r \dd \theta \dd z\\
   &-\int_{0}^{1}\int_{0}^{2\pi}\int_{R-1}^{R}\left(\frac{v^{\theta}\p_{\theta}W^{z}}{r}+v^{z}\p_{z}W^{z}+W^{r}\p_{r}v^{z}+\frac{W^{\theta}\p_{\theta}v^{z}}{r}+W^{z}\p_{z}v^{z}\right)\Psi_{R}^{z}\,\dd r \dd \theta \dd z.
    \end{split}
    \end{equation}
  Now we are ready to estimate the right hand side of   \eqref{prEstim}-\eqref{prepz}. For the right hand side of \eqref{prEstim}, by \eqref{eqPrPsies} and the Poincar\'e inequality,
  one has
  \begin{equation}\label{prvrpsir}
  \begin{split}
&\left|\int_{0}^{1}\int_{0}^{2\pi}\int_{R-1}^{R} (\p_{r}v^{r}\p_{r}\Psi_{R}^{r}+\p_{z}v^{r}\p_{z}\Psi_{R}^{r}) \,\dd r \dd \theta \dd z\right|\\
\leq&\, C\|(\p_{r}v^{r},\p_{z}v^{r})\|_{L^{2}(\DR)}\|(\p_{r}\Psi_{R}^{r},\p_{z}\Psi_{R}^{r})\|_{L^{2}(\DR)}\\
\leq &\,CR^{-\frac{1}{2}}\|\nabla\Bv\|_{L^{2}(\OR)}\cdot R^{\frac{1}{2}}\|v^{r}\|_{L^{2}(\OR)}\\
\leq&\, C\|\nabla\Bv\|^{2}_{L^{2}(\OR)}
\end{split}
  \end{equation}
  and 
\begin{equation}
\begin{split}
&\left|\int_{0}^{1}\int_{0}^{2\pi}\int_{R-1}^{R} \left(\frac{1}{r^{2}}\p_{\theta}v^{r}-\frac{2}{r^{2}}v^{\theta}\right)\p_{\theta}\Psi_{R}^{r}\,\dd r \dd \theta \dd z\right|\\
 \leq&\,  C(\|r^{-2}(\p_{\theta}v^{r}-v^{\theta})\|_{L^{2}(\DR)}+\|r^{-2}v^{\theta}\|_{L^{2}(\DR)})\cdot R^{\frac{1}{2}}\|v^{r}\|_{L^{2}(\OR)}\\
  \leq&\, C(R^{-\frac{3}{2}}\|\nabla\Bv\|_{L^{2}(\OR)}+R^{-\frac{5}{2}}\|\nabla\Bv\|_{L^{2}(\OR)})\cdot R^{\frac{1}{2}}\|\nabla\Bv\|_{L^{2}(\OR)}\\
   \leq&\, C R^{-1}\|\nabla\Bv\|^{2}_{L^{2}(\OR)},
 \end{split}
\end{equation}
where the third line is due to \eqref{pthetav}. Furthermore, by the Gagliardo-Nirenberg inequality, one has 
\begin{equation}
    \begin{split}
        &\left|\int_{0}^{1}\int_{0}^{2\pi}\int_{R-1}^{R}\left(\frac{1}{r}\p_{r}v^{r}-\frac{1}{r^{2}}v^{r}\right)\Psi_{R}^{r}\,\dd r \dd \theta \dd z\right|\\
        \leq&\,CR^{-1}\|\nabla\Bv\|_{L^{2}(\DR)} \|\p_{z}\Psi_{R}^{r}\|_{L^{2}(\DR)}\\
        \leq&\,CR^{-\frac{3}{2}}\|\nabla\Bv\|_{L^{2}(\OR)} \cdot R^{\frac{1}{2}}\|v^{r}\|_{L^{2}(\OR)}\\
       \leq&\,CR^{-1} \|\nabla\Bv\|^{2}_{L^{2}(\OR)}, 
    \end{split}
\end{equation}
\begin{equation}\label{completer}
    \begin{split}
&\left|\int_{0}^{1}\int_{0}^{2\pi}\int_{R-1}^{R}\left[\left(v^r \partial_r+ \frac{v^\theta}{r}  \partial_\theta+v^z \partial_z\right)v^r-\frac{(v^{\theta})^{2}}{r}\right]\Psi_{R}^{r}\,\dd r \dd \theta \dd z\right|\\
\leq&\,C\|(v^{r},v^{\theta},v^{z})\|_{L^{3}(\DR)}(\|(\p_{r},\p_{z})v^{r}\|_{L^{2}(\DR)}+\|r^{-1}(\p_{\theta}v^{r}-v^{\theta}\|_{L^{2}(\DR)})\|\Psi_{R}^{r}\|_{L^{6}(\DR)}\\
\leq&\,C\|\Bv\|^{\frac{1}{2}}_{L^{2}(\DR)}(\|\Bv\|^{\frac{1}{2}}_{L^{2}(\DR)}+\|\overline{\nabla}\Bv\|^{\frac{1}{2}}_{L^{2}(\DR)})\cdot
R^{-\frac{1}{2}}\|\nabla\Bv\|_{L^{2}(\OR)}\|\overline{\nabla}\Psi_{R}^{r}\|_{L^{2}(\DR)}\\
\leq&\,CR^{-\frac{1}{4}}\|\nabla\Bv\|^{\frac{1}{2}}_{L^{2}(\OR)}(R^{\frac{1}{4}}\|\nabla\Bv\|^{\frac{1}{2}}_{L^{2}(\OR)}+R^{-\frac{1}{4}}\|\Bv\|^{\frac{1}{2}}_{L^{2}(\OR)})\cdot R^{-\frac{1}{2}}\|\nabla\Bv\|_{L^{2}(\OR)}\cdot R^{\frac{1}{2}}\|\nabla\Bv\|_{L^{2}(\OR)}\\
\leq&\,C\|\nabla\Bv\|^{3}_{L^{2}(\OR)}
\end{split}
\end{equation}
and
\begin{equation}\label{estimatesmall}
    \begin{split}
&\left|\int_{0}^{1}\int_{0}^{2\pi}\int_{R-1}^{R}\left(\p_{r}(v^{r} W^{r})-\frac{2v^{\theta}W^{\theta}}{r}+\frac{v^{\theta}\partial_{\theta}W^{r}}{r}+v^{z}\partial_{z}W^{r}+\frac{W^{\theta}\partial_{\theta}v^r}{r}+W^{z}\partial_{z}v^{r}\right)\Psi_{R}^{r}\,\dd r \dd \theta \dd z\right|\\
&\leq C\|\bBW\|_{W^{1,\infty}(\OR)}(\|(\p_{r},\p_{z})\Bv\|_{L^{2}(\DR)}+\|r^{-1}(\p_{\theta}v^{r}-v^{\theta})\|_{L^{2}(\DR)})\|\p_{z}\Psi_{R}^{r}\|_{L^{2}(\DR)}\\
&\leq C\|\bBW\|_{W^{1,\infty}(\OR)} R^{-\frac{1}{2}}\|\nabla\Bv\|_{L^{2}(\OR)}\cdot
R^{\frac{1}{2}}\|v^{r}\|_{L^{2}(\OR)}\\
&\leq C\|\nabla\Bv\|^{2}_{L^{2}(\OR)},
  \end{split}
\end{equation}
where the last inequality is due to $\|\bBW\|_{W^{1,\infty}(\OR)}\leq C$.
Combining \eqref{prvrpsir}-\eqref{estimatesmall}, it holds that
\begin{equation}\label{estima4}
   \left| \int_{0}^{1}\int_{0}^{2\pi}\int_{R-1}^{R}\p_{r}P\Psi_{R}^{r} \,\dd r \dd \theta \dd z\right|\leq C\|\nabla\Bv\|^{2}_{L^{2}(\OR)}+C\|\nabla\Bv\|^{3}_{L^{2}(\OR)}.
\end{equation}
As for the right-hand side of \eqref{pthetaes}, using \eqref{eqPrPsies},
\eqref{prvrpzv}-\eqref{pthetav}, one has
\begin{equation}\label{thetater}
    \begin{split}
        &\left|\int_{0}^{1}\int_{0}^{2\pi}\int_{R-1}^{R} r(\p_{r}v^{\theta}\p_{r}\Psi_{R}^{\theta}+\p_{z}v^{\theta}\p_{z}\Psi_{R}^{\theta}) \,\dd r \dd \theta \dd z\right|\\
     \leq&\, CR\|(\p_{r}v^{\theta},\p_{z}v^{\theta})\|_{L^{2}(\DR)}\|(\p_{r}\Psi_{R}^{\theta},\p_{z}\Psi_{R}^{\theta})\|_{L^{2}(\DR)}\\
    \leq &\,CR^{\frac{1}{2}}\|\nabla\Bv\|_{L^{2}(\OR)}\cdot R^{\frac{1}{2}}\|v^{r}\|_{L^{2}(\OR)}\\
\leq&\, CR\|\nabla\Bv\|^{2}_{L^{2}(\OR)},
    \end{split}
\end{equation}
\begin{equation}
    \begin{split}
    &\left| \int_{0}^{1}\int_{0}^{2\pi}\int_{R-1}^{R}r^{-1}(\p_{\theta}v^{\theta}\p_{\theta}\Psi_{R}^{\theta}+2v^{r}\p_{\theta}\Psi_{R}^{\theta})\,\dd r \dd \theta \dd z\right|\\
  \leq&\,C(\|r^{-1}(\p_{\theta}v^{\theta}+v^{r})\|_{L^{2}(\DR)}+R^{-1}\|v^{r}\|_{L^{2}(\DR)})\|\p_{\theta}\Psi_{R}^{\theta}\|_{L^{2}(\DR)}\\
  \leq&\,CR^{-\frac{1}{2}}\|\nabla\Bv\|_{L^{2}(\OR)}\cdot R^{\frac{1}{2}}\|v^{r}\|_{L^{2}(\OR)}\\
   \leq&\,C\|\nabla\Bv\|^{2}_{L^{2}(\OR)}
       \end{split}
\end{equation}
and
\begin{equation}
    \begin{split}
 &\left|\int_{0}^{1}\int_{0}^{2\pi}\int_{R-1}^{R}r^{-1}v^{\theta}\Psi_{R}^{\theta}\,\dd r \dd \theta \dd z \right|\\
 \leq&\,CR^{-1}\|\p_{z}v^{\theta}\|_{L^{2}(\DR)}\|\p_{z}\Psi_{R}^{\theta}\|_{L^{2}(\DR)}\\
 \leq&\,CR^{-\frac{3}{2}}\|\nabla\Bv\|_{L^{2}(\OR)}\cdot R^{\frac{1}{2}}\|v^{r}\|_{L^{2}(\OR)}\\
\leq&\,CR^{-1} \|\nabla\Bv\|^{2}_{L^{2}(\OR)}.
   \end{split} 
\end{equation}
Moreover, similar to \eqref{completer} and \eqref{estimatesmall},  one obtains
\begin{equation}
\begin{split}
&\left|\int_{0}^{1}\int_{0}^{2\pi}\int_{R-1}^{R}\left[r\left(v^r \partial_r+ \frac{v^\theta}{r}  \partial_\theta+v^z \partial_z\right)v^{\theta}+v^{\theta}v^{r}\right]\Psi_{R}^{\theta}\,\dd r \dd \theta \dd z\right|\\
\leq&\,CR\|(v^{r},v^{\theta},v^{z})\|_{L^{3}(\DR)}(\|(\p_{r},\p_{z})v^{\theta}\|_{L^{2}(\DR)}+\|r^{-1}(\p_{\theta}v^{\theta}+v^{r})\|_{L^{2}(\DR)})\|\Psi_{R}^{\theta}\|_{L^{6}(\DR)}\\
\leq&\,CR\|\nabla\Bv\|^{3}_{L^{2}(\OR)},
\end{split}
\end{equation}

\begin{equation}
\begin{split}
&\left|\int_{0}^{1}\int_{0}^{2\pi}\int_{R-1}^{R}(rv^{r}\p_{r}W^{\theta}+v^{\theta}\p_{\theta}W^{\theta}+W^{r}v^{\theta}+rv^{z}\p_{z}W^{\theta})\Psi_{R}^{\theta}\,\dd r \dd \theta \dd z\right|\\
\leq&\,CR\|\bBW\|_{W^{1,\infty}(\OR)}\|\p_{z}(v^{r},v^{\theta},v^{z})\|_{L^{2}(\DR)}\|\p_{z}\Psi_{R}^{\theta}\|_{L^{2}(\DR)}\\
\leq&\,CR\|\bBW\|_{W^{1,\infty}(\OR)}\|\nabla\Bv\|_{L^{2}(\OR)}\|v^{r}\|_{L^{2}(\OR)}\\
\leq&\,CR\|\nabla\Bv\|^{2}_{L^{2}(\OR)}
\end{split}
\end{equation}
and
\begin{equation}\label{thetafina}
    \begin{split}
 &\left|\int_{0}^{1}\int_{0}^{2\pi}\int_{R-1}^{R}(rW^{r}\p_{r}v^{\theta}+W^{\theta}\p_{\theta}v^{\theta}+W^{\theta}v^{r}+rW^{z}\p_{z}v^{\theta})\Psi_{R}^{\theta}\,\dd r \dd \theta \dd z\right|\\  
 \leq&\,CR\|\bBW\|_{L^{\infty}(\OR)}(\|(\p_{r},\p_{z})v^{\theta}\|_{L^{2}(\DR)}+\|r^{-1}(\p_{\theta}v^{\theta}+v^{r})\|_{L^{2}(\DR)})\|\p_{z}\Psi_{R}^{\theta}\|_{L^{2}(\DR)}\\
 \leq&\,CR\|\bBW\|_{L^{\infty}(\OR)}\|\nabla\Bv\|_{L^{2}(\OR)}\|v^{r}\|_{L^{2}(\OR)} \\
 \leq&\,CR\|\nabla\Bv\|^{2}_{L^{2}(\OR)}.
    \end{split}
\end{equation}
Collecting the estimates \eqref{thetater}-\eqref{thetafina} gives
\begin{equation}\label{estima5}
 \left| \int_{0}^{1}\int_{0}^{2\pi}\int_{R-1}^{R}\p_{\theta}P\Psi_{R}^{\theta} \,\dd r \dd \theta \dd z\right|\leq CR\|\nabla\Bv\|^{2}_{L^{2}(\OR)}+CR\|\nabla\Bv\|^{3}_{L^{2}(\OR)}.
\end{equation}
Similarly, it can be shown  that
\begin{equation}\label{estima6}
\left|\int_{0}^{1}\int_{0}^{2\pi}\int_{R-1}^{R}\p_{z}P\Psi_{R}^{z} \,\dd r \dd \theta \dd z\right|\leq C\|\nabla\Bv\|^{2}_{L^{2}(\OR)}+C\|\nabla\Bv\|^{3}_{L^{2}(\OR)}.
\end{equation}
 From the estimate \eqref{initieqf}-\eqref{estima2}, \eqref{estima3}, \eqref{estima4}, \eqref{estima5}-\eqref{estima6}, one arrives at 
 \begin{equation}\label{fianeqes}
\int_{\Omega}|\nabla\Bv|^{2}\varphi_R\,\dd \Bx\leq CR\|\nabla\Bv\|^{2}_{L^{2}(\OR)}+CR\|\nabla\Bv\|^{3}_{L^{2}(\OR)} \le C_1 R\|\nabla\Bv\|^{2}_{L^{2}(\OR)},
 \end{equation}
 where in  the last inequality we have used  the  bound $\Bv \in H^1(\Omega)$, where established in Step 1.
 
 Let 
 \[
 Y(R)=\int_{0}^{1}\iint_{\mathbb{R}^{2}}|\nabla\Bv|^{2}\varphi_R\left(\sqrt{x^{2}_{1}+x^{2}_{2}}\right)\,\dd x_{1} \dd x_{2} \dd x_{3}.
 \]
The straightforward computations give
 \[
 Y^{\prime}(R)=-\int_{0}^{1}\int_{0}^{2\pi}\int_{R-1}^{R}|\nabla\Bv|^{2}r\,\dd r \dd \theta \dd z=-\int_{\OR}|\nabla\Bv|^{2}\,\dd \Bx.
 \]
 Hence \eqref{fianeqes} can be written as
 \begin{equation}
     Y(R)\leq -C_1RY^{\prime}(R).
 \end{equation}

For every fixed large $R_{1}$, integrating this inequality over $[R_{1},R]$ yields
\begin{equation}
 Y(R)\leq R_{1}^{\frac{1}{C_{1}}}Y(R_{1})R^{-\frac{1}{C_{1}}}
 \leq C_{2}R^{-\frac{1}{C_{1}}}.
\end{equation}		
Combining the Poincar\'e inequality	yields 
\begin{equation}
    \|\Bu-\bBW\|_{H^1(\Omega_{>R})} = \|\Bv\|_{H^1(\Omega_{>R})} \le C[Y(R)]^{1/2} \le C_2^{\frac{1}{2}} R^{-\frac{1}{2C_1}}.
\end{equation}
Taking $\alpha=\frac{1}{2C_1}>0$, we arrive at the desired estimate \eqref{eq:prop4.1_decay}. This completes the proof of Proposition \ref{prop4.1}.
	\end{proof}

We now proceed to improve the regularity of the perturbation field $\Bv$ to obtain the pointwise convergence rate.

\begin{pro}\label{prop4.2}
Let $\Bu = \bBW + \Bv$ be the  generalized weak solution obtained in Proposition \ref{prop4.1}. Then  the perturbation field $\Bv$ satisfies the following pointwise decay estimate in the far field:
\begin{equation}\label{eq:prop4.2_decay}
|\Bu(\Bx) - \bBW(x_3)| \le C |x'|^{-\alpha}, \quad \forall \Bx= (x_1, x_2, x_3) \in \Omega \text{ with } |x'| \ge R_1,
\end{equation}
where $x'=(x_1, x_2)$, and the  constant $C > 0$ depends only on $\Omega$ and $F$.
\end{pro}

\begin{proof}[Proof for Proposition  \ref{prop4.2}]
The proof relies on the  elliptic regularity theory for the Stokes equations. We employ a bootstrap argument to upgrade the control of the perturbation $\Bv$ from the $H^1$-norm to the $L^\infty$-norm. For any far-field point $x^* = (x_1^*, x_2^*, x_3^*) \in \Omega$ with $R = |x'^*| \ge R_1$, we divide the analysis into two cases based on its distance to the boundary $\partial\Omega$.

\emph{Step 1.} \emph{Interior Estimates.}
Assume that $\text{dist}(x^*, \partial\Omega) \ge 1/4$. We consider two fixed concentric balls centered at $x^*$, denoted as $B_{1/8}(x^*) \subset B_{1/4}(x^*) \subset \Omega$. For sufficiently large $R_1$, the larger ball $B_{1/4}(x^*)$ lies  within the far-field exterior domain $\Omega_{>R/2}=\{(x_{1},x_{2})\in\mathbb{R}^{2}:|x^{\prime}|>R/2\}\times(0,1)$. 

Applying Lemma \ref{interior regularity} with $\Omega'=B_{1/8}(x^*)$ and $ \Omega''=B_{1/6}(x^*) $. In particular, taking $q = \frac{5}{3}$ and $m = 0$ in \eqref{6-5} and using Sobolev embedding inequality, we have
\begin{equation}\label{eq:interior_stokes}
    \begin{split}  
    &\|\Bv\|_{L^{\infty}(B_{1/8}(x^*))}   \\
\le &\, \|\Bv\|_{W^{2,\frac{5}{3}}(B_{1/8}(x^*))}\\
\le&\, C \left( \|\Bv \cdot \nabla \Bv + \bBW \cdot \nabla \Bv + \Bv \cdot \nabla \bBW\|_{L^{\frac{5}{3}}(B_{1/6}(x^*))} + \|\Bv\|_{W^{1,\frac{5}{3}}(B_{1/6}(x^*))} \right) \\
\le &\,C \left( \|\Bv\|_{L^{10}(B_{1/6}(x^*))} \|\nabla \Bv\|_{L^2(B_{1/6}(x^*))} + (\|\bBW\|_{W^{1,\infty}(\Omega)} + 1) \|\Bv\|_{W^{1,2}(B_{1/6}(x^*))} \right).
 \end{split}
\end{equation}
To further bound the remaining localized $L^{10}$-norm, we shift the integration exponents and expand the subdomains to a wider radius by  Lemma \ref{interior regularity} once again, which yields
    \begin{equation}\label{eqL10ne2}
        \begin{split}
       &\|\Bv\|_{L^{10}(B_{1/6}(x^*))}\\ \leq\,& C\|\Bv\|_{W^{2,\frac{30}{23}}(B_{1/4}(x^*))} \\
       \leq \,&  C\left( \|\Bv \cdot \nabla \Bv + \bBW \cdot \nabla \Bv + \Bv \cdot \nabla \bBW\|_{L^{\frac{30}{23}}(B_{1/4}(x^*))}+\|\Bv\|_{W^{1,\frac{30}{23}}(B_{1/4}(x^*))}\right)\\  
       \leq \,&  C\left( \|\Bv\|_{L^{\frac{15}{4}}(B_{1/4}(x^*))}\|\nabla\Bv\|_{L^{2}(B_{1/4}(x^*))}+(\|\bBW\|_{W^{1,\infty}(\Omega)} + 1) \|\Bv\|_{W^{1,\frac{30}{23}}(B_{1/4}(x^*))}\right)\\
        \leq \,& C\left(\|\nabla\Bv\|_{L^{2}(B_{1/4}(x^*))}+\|\bBW\|_{W^{1,\infty}(\Omega)} + 1\right) \|\Bv\|_{W^{1,2}(B_{1/4}(x^*))}.
        \end{split}
    \end{equation}
Utilizing the estimate \eqref{eqL10ne2} and $\|\bBW\|_{W^{1,\infty}(\Omega)}$ is bounded, \eqref{eq:interior_stokes} yields
\begin{equation}
\begin{aligned}
&\|\Bv\|_{L^{\infty}(B_{1/8}(x^*))}\\
\le &\, \|\Bv\|_{W^{2,\frac{5}{3}}(B_{1/8}(x^*))}\\
 \le&\, C \left\{  \|\nabla \Bv\|_{L^2(B_{1/4}(x^*))}(\|\nabla \Bv\|_{L^2(B_{1/4}(x^*))}+1)\|\Bv\|_{W^{1,2}(B_{1/4}(x^*))}+\|\Bv\|_{H^1(B_{1/4}(x^*))} \right\}\\ 
\le&\, C \left( 1+\|\Bv\|_{H^1(B_{1/4}(x^*))}  + \|\Bv\|_{H^1(B_{1/4}(x^*))}^2 \right)\cdot \|\Bv\|_{H^1(B_{1/4}(x^*))}\\
\le &\,C \|\Bv\|_{H^1(B_{1/4}(x^*))} .
\end{aligned}
\end{equation}
Recalling the energy decay estimate \eqref{eq:prop4.1_decay} from Proposition \ref{prop4.1}  and since $B_{1/4}(x^*) \subset \Omega_{>R/2}$, we conclude that
\begin{equation}
|\Bv(x^*)| \le \|\Bv\|_{L^\infty(B_{1/8}(x^*))} \le C \|\Bv\|_{H^1(\Omega_{>R/2})} \le C \left(\frac{R}{2}\right)^{-\alpha} \le C |x'^*|^{-\alpha}.
\end{equation}

\emph{Step 2.} \emph{Boundary Estimates.} Now assume that $\text{dist}(x^*, \partial\Omega) < 1/4$. Without loss of generality, suppose $x^*$ is close to the upper boundary $x_3 = 1$. Let $x_0 = (x_1^*, x_2^*, 1) \in \partial\Omega$ be the vertical projection of $x^*$ onto the boundary. It is clear that $x^* \in B_{1/4}^+(x_0) = B_{1/4}(x_0) \cap \Omega$.

We introduce a fixed, smooth cut-off function $\zeta \in C_0^\infty(B_{1/2}(x_0))$ such that $\zeta \equiv 1$ on $B_{1/4}(x_0)$. Since the size of this localized neighborhood is chosen as a fixed universal constant $1/2$, the derivatives of $\zeta$ satisfy uniform bounds $|\nabla \zeta| \le C$ and $|\nabla^2 \zeta| \le C$, which are completely independent of the distance from $x^*$ to the boundary.

Let $\hat{\Bv}(\Bx) = \zeta(\Bx) \Bv(\Bx)$ and $\hat{P}(\Bx) = \zeta(\Bx) P(\Bx)$. Extending $\hat{\Bv}$ and $\hat{P}$ by zero to the entire half-space $\mathbb{R}^3_+$, they satisfy the following standard non-homogeneous Stokes equations:
\begin{equation}\label{eq:boundary_stokes_sys}
\begin{cases}
-\Delta \hat{\Bv} + \nabla \hat{P} = \BF & \text{in } \mathbb{R}^3_+, \\
\mathrm{div} \, \hat{\Bv} = G & \text{in } \mathbb{R}^3_+, \\
\hat{\Bv} = 0 & \text{on } \{x_3 = 1\},
\end{cases}
\end{equation}
where the source and correction terms are given by
\begin{align}
\BF &= -\Bv \Delta \zeta - 2 \nabla \zeta \cdot \nabla \Bv + P \nabla \zeta - \zeta (\Bv \cdot \nabla \Bv + \bBW \cdot \nabla \Bv + \Bv \cdot \nabla \bBW), \\
G &= \Bv \cdot \nabla \zeta.
\end{align}

Inspired by the standard regularity theory for the Stokes equations in the half-space Lemma \ref{bpm} and eliminating the local pressure term on the right-hand side by modifying $\hat{P}$ by an appropriate constant (as  justified in Remark 2.1), we obtain the boundary estimate analogue of the local elliptic estimate:
\begin{equation}\label{eq:boundary_w253}
\|\Bv\|_{W^{2,\frac{5}{3}}(B_{1/4}^+(x_0))} \le C \left( \|\Bv \cdot \nabla \Bv + \bBW \cdot \nabla \Bv + \Bv \cdot \nabla \bBW\|_{L^{\frac{5}{3}}(B_{1/2}^+(x_0))} + \|\Bv\|_{W^{1,\frac{5}{3}}(B_{1/2}^+(x_0))} \right).
\end{equation}

Following the exact same H\"older's inequality and Sobolev embedding arguments as presented in Step 1, and noting that the larger fixed half-ball $B_{1/2}^+(x_0)$ is strictly contained within the far-field exterior domain $\Omega_{>R/2}$ (since $\frac{R}{2} \ge R_1$), the right-hand side of \eqref{eq:boundary_w253} can be  controlled by the $H^1$-energy norm:
\begin{equation}
\|\Bv\|_{W^{2,\frac{5}{3}}(B_{1/4}^+(x_0))} \le C \|\Bv\|_{H^1(B_{1/2}^+(x_0))} \le C \|\Bv\|_{H^1(\Omega_{>R/2})} \le C \left(\frac{R}{2}\right)^{-\alpha} \le C |x'^*|^{-\alpha}.
\end{equation}

Finally, applying the three-dimensional Sobolev embedding theorem $W^{2,\frac{5}{3}}(B_{1/4}^+(x_0)) \hookrightarrow L^\infty(B_{1/4}^+(x_0))$, we achieve the pointwise control at the boundary:
\begin{equation}
|\Bv(x^*)| \le \|\Bv\|_{L^\infty(B_{1/4}^+(x_0))} \le C |x'^*|^{-\alpha}.
\end{equation}
This completes the proof of Proposition \ref{prop4.2}.
\end{proof}

\section{uniqueness for solutions with small force}\label{Sec4}

In this section, we establish the   uniqueness of the finite-energy perturbation field $\Bv = \Bu - \bBW$, which completes the proof of Theorem 1.1. As discussed in Section \ref{Sec3}, while the existence and far-field decay of the solution only require the compact support of the external force, its uniqueness is more restrictive. To strictly control the nonlinear convection terms, an additional smallness assumption on $
\Bf$ is required. We now rigorously justify this uniqueness property.

\begin{pro}\label{prop3.1}
Let  $\Bu = \bBW + \Bv$ be the weak solution obtained in  Proposition \ref{prop4.1}. There exists a  small constant $\epsilon_1 > 0$ such that if the external force additionally satisfies  $\|\Bf\|_{V^{\prime}} \le \epsilon_1$, then $\Bu$ is unique in the class of finite-energy perturbations. Specifically, if $\BU_{sol} = \bBW + \BU$ is another weak solution to the Navier-Stokes equations \eqref{eqsteadynfsf}  such that its perturbation satisfies $\BU \in H^{1}(\Omega)$, then $\Bu = \BU_{sol}$.
\end{pro}

\begin{proof}[Proof for Proposition  \ref{prop3.1}]

 Let $\Bw = \Bv - \BU$ and $q = p_{\Bv} - p_{\BU}$, so that the  $\Bw$ satisfies:
\begin{equation}\label{eq:uniq_w_final}
-\Delta \Bw + (\Bw\cdot\nabla)\Bv + (\BU\cdot\nabla)\Bw + (\bBW\cdot\nabla)\Bw+ (
\Bw\cdot\nabla)\bBW + \nabla q = 0.
\end{equation}
Since both $\Bv$ and $\BU$ belong to $H^1(\Omega)$, their difference $\Bw$ belongs to $H^1(\Omega)$. 


Multiplying \eqref{eq:uniq_w_final} by $\Bw$ and integrating over the slab $\Omega$, the terms $\int_{\Omega} (\BU\cdot\nabla)\Bw\cdot \Bw \,\dd \Bx$,\linebreak and $\int_{\Omega} (\bBW\cdot\nabla)\Bw\cdot \Bw \,\dd \Bx$     vanish identically due to the incompressibility $\nabla \cdot \BU = \nabla \cdot \bBW = 0$ and the no-slip boundary conditions. This implies
\begin{equation}\label{eq:uniq_energy}
\int_{\Omega} |\nabla \Bw|^2 \, \dd \Bx = - \int_{\Omega} (\Bw\cdot\nabla)\Bv \cdot \Bw \,\dd \Bx - \int_{\Omega} (\Bw\cdot\nabla)\bBW \cdot \Bw \, \dd \Bx.
\end{equation}
Integrating by parts and applying the  Hölder's inequality and Sobolev embeddings $H^1(\Omega) \hookrightarrow L^3(\Omega) \cap L^6(\Omega)$  to the right-hand side of \eqref{eq:uniq_energy}, we obtain
\begin{align*}
\left| \int_{\Omega} (\Bw\cdot\nabla)\Bv \cdot \Bw \,\dd \Bx \right| &\le \|\Bw\|_{L^6(\Omega)} \|\nabla \Bw\|_{L^2(\Omega)} \|\Bv\|_{L^3(\Omega)} \le C\|\Bv\|_{H^1(\Omega)} \|\nabla \Bw\|_{L^2(\Omega)}^2, \\
\left| \int_{\Omega} (\Bw\cdot\nabla)\bBW \cdot \Bw \, \dd \Bx \right| &\le \|\Bw\|_{L^2(\Omega)} \|\nabla \Bw\|_{L^2(\Omega)} \|\nabla \bBW\|_{L^\infty(\Omega)} \le C \|\nabla \bBW\|_{L^\infty(\Omega)} \|\nabla \Bw\|_{L^2(\Omega)}^2.
\end{align*}
Combining these bounds into \eqref{eq:uniq_energy} yields:
\begin{equation}\label{cp}
\left( 1 - C \|\Bv\|_{H^1(\Omega)} - C\|\nabla \bBW\|_{L^\infty(\Omega)} \right) \|\nabla \Bw\|_{L^2(\Omega)}^2 \le 0.
\end{equation}
Since the  perturbation $\Bv$  satisfies $\|\Bv\|_{H^1(\Omega)} \le C\|\Bf\|_{V^{\prime}} \le C\epsilon_1 $ from the proof in Proposition \ref{prop4.1}, and the Poiseuille flow satisfies $\|\nabla \bBW\|_{L^\infty(\Omega)} \ll 1$ (since 
 $|F| \le \epsilon_0$), the coefficient on the left-hand side of \eqref{cp} is strictly positive.
Consequently, we obtain $\|\nabla \Bw\|_{L^2(\Omega)} = 0$. In view of the no-slip boundary conditions, this implies $\Bw \equiv 0$, which means $\BU \equiv \Bv$. The proof of Proposition \ref{prop3.1} is completed.
\end{proof}

	{\bf Acknowledgement.}  J. Han was supported by  Start-up funds for doctoral research of Anhui Normal University (No. 762610).
	The authors thank Professor Yun Wang for valuable discussions and suggestions.

\end{document}